\documentclass[a4paper, 12pt]{article}

\usepackage{amsfonts, amsmath, amsthm}
\usepackage{amssymb}
\usepackage{latexsym}
\usepackage[dvips]{graphicx}
\usepackage{color}

\theoremstyle{plain}
\newtheorem{thm}{Theorem}[section]

\newtheorem{lem}{Lemma}[section]

\theoremstyle{definition}

\theoremstyle{definition}
\newtheorem{rmk}{Remark}[section]

\numberwithin{equation}{section}
\newcommand{\Z}{\mathbb{Z}}

\newcommand{\R}{\mathbb{R}}

\newcommand{\Sph}{\mathbb{S}}

\newcommand{\pa}{\partial}
\newcommand{\eps}{\varepsilon}
\newcommand{\supp}{{\rm supp\,}}
\newcommand{\jb}[1]{\langle #1 \rangle}

\newcommand{\dis}{\displaystyle}

\begin{document}
%%%%%%%%%%%%%%%%%%%%%%%%%%%%%%%%%%%%%%%%%%%%%%%%%%%%%%%%%%%%%%%%%%%%%%%%%%%%%%%
\title{
Energy decay for small solutions to semilinear wave equations 
with weakly dissipative structure
 }

%-----------------------
\author{
          %\underline{Yoshinori Nishii} \thanks{
          Yoshinori Nishii\thanks{
              Department of Mathematics, Graduate School of Science, 
              Osaka University. 
              1-1 Machikaneyama-cho, Toyonaka, 
              Osaka 560-0043, Japan. 
              (E-mail: {\tt y-nishii@cr.math.sci.osaka-u.ac.jp}) 
            }
           \and  
          %\underline{Hideaki Sunagawa} \thanks{
          Hideaki Sunagawa \thanks{
         Department of Mathematics, Graduate School of Science, 
              Osaka University. 
              1-1 Machikaneyama-cho, Toyonaka, 
              Osaka 560-0043, Japan. 
               (E-mail: {\tt sunagawa@math.sci.osaka-u.ac.jp})
             }
           \and  
          %\underline{Hiroki Terashita} \thanks{
          Hiroki Terashita \thanks{
              Department of Mathematics, Graduate School of Science, 
              Osaka University. 
              1-1 Machikaneyama-cho, Toyonaka, 
              Osaka 560-0043, Japan. 
              (E-mail: {\tt u233599e@ecs.osaka-u.ac.jp}) 
             }
}%-----------------------

\date{\today }   
\maketitle

\vspace{-4mm}
\begin{center}
{\em Dedicated to the memory of Professor K\^oji Kubota}
\end{center}
\vspace{2mm}

%-----------------------
\noindent{\bf Abstract:}\ 
This article gives  an energy decay result for small data solutions to 
a class of semilinear wave equations in two space dimensions 
possessing weakly dissipative structure relevant to 
the Agemi condition.
\\

%-----------------------
\\

%-----------------------
\\

%-----------------------
%\noindent{\bf Running title:}\ 
%Semilinear wave equations with weakly dissipative structure.

%%%%%%%%%%%-----------------------------------------------------------------
\section{Introduction and the result}  \label{sec_intro}
%%%%%%%%%%%-----------------------------------------------------------------
This article deals with large-time behavior of the solution $u=u(t,x)$ to 
the Cauchy problem 
\begin{align}
& \Box  u= F(\pa u), && (t,x) \in (0, \infty) \times \R^2,
\label{Equ}\\
& u(0,x)=\eps f(x), \ \pa_t u(0,x)=\eps g(x), && x\in \R^2,
 \label{Data}
\end{align}
where  $\Box=\pa_t^2-\Delta=\pa_t^2-(\pa_1^2+\pa_2^2)$ 
and $\pa u=(\pa_0 u, \pa_1 u, \pa_2 u)$
with 
$\pa_0=\pa_t=\pa/\pa t$, $\pa_{1}=\pa/ \pa x_1$, $\pa_{2}=\pa/ \pa x_2$ 
for $t\ge 0$, $x=(x_1,x_2)\in \R^2$.
We suppose that $f$, $g:\R^{2}\to \R$ are compactly-supported 
$C^{\infty}$-functions,  
$\eps$ is a small positive parameter, and $F:\R^{1+2}\to \R$ 
is a smooth function vanishing of quadratic order in a neighborhood of 
$0\in \R^{1+2}$. More explicitly, we suppose that 
\[
 F(\pa u)=F_{\rm q}(\pa u)+F_{\rm c}(\pa u)+O(|\pa u|^4)
\]
near $\pa u=0$, where the quadratic homogeneous part 
$F_{\rm q}(\pa u)$ and the cubic homogeneous part $F_{\rm c}(\pa u)$ 
are given by
\begin{align}
 F_{\rm q}(\pa u)
 = 
 \sum_{j,k=0}^2 B_{jk} (\pa_j u)(\pa_k u),
\quad
 F_{\rm c}(\pa u)
 = 
 \sum_{j,k,l=0}^{2}  C_{jkl} (\pa_j u)(\pa_k u)(\pa_l u)
\label{expression_nonli}
\end{align}
with some real constants $B_{jk}$ and $C_{jkl}$, respectively. 

Before going into the details, let us summarize some of the known results 
briefly to make our motivation clear. 
From the viewpoint of nonlinear perturbation, two-dimensional case is  of 
special interest for the wave equations 
because the essential nonlinear effects can be observed in the large-time 
behavior of the solutions even if the initial data are sufficiently small. 
If both $F_{\rm q}$ and $F_{\rm c}$ are absent and $\eps$ is suitably small, 
we can show global existence of a unique $C^{\infty}$-solution for 
\eqref{Equ}--\eqref{Data} by the standard way 
(see e.g., Chapter~6 in \cite{Kat2}), 
however, we cannot expect the same conclusion for general $F_{\rm q}$ and 
$F_{\rm c}$. In other words, we need some structural conditions for the 
quadratic and cubic parts of $F$ if we expect the small data global existence 
for \eqref{Equ}--\eqref{Data}.  
One of the most famous structural conditions may be the so-called null 
conditions, which we explain now. 
In what follows $\Sph^1$ stands for the unit circle in $\R^2$, and 
we write $\hat{\omega}=(\omega_0,\omega_1,\omega_2)$ for 
$\omega=(\omega_1,\omega_2) \in \Sph^1$ with the convention $\omega_0=-1$.
We say that the quadratic (resp. cubic) null condition is satisfied 
if  $F_{\rm q}(\hat{\omega})$ (resp.  $F_{\rm c}(\hat{\omega})$) vanishes 
identically on $\Sph^1$. Remember that
\begin{align*}
 F_{\rm q}(\hat{\omega})
 =  \sum_{j,k=0}^2 B_{jk}\omega_j\omega_k,
 \quad
 F_{\rm c}(\hat{\omega})
 =  \sum_{j,k,l=0}^{2}  C_{jkl}  \omega_j \omega_k \omega_l
\end{align*}
with the coefficients $B_{jk}$ and $C_{jkl}$ appearing in 
\eqref{expression_nonli}.
According to the classical result by Godin~\cite{Go} 
(see also \cite{Ali1a}, \cite{Hos1}, \cite{Kat1}, \cite{Zha}, etc.), 
the Cauchy problem \eqref{Equ}--\eqref{Data} admits a unique global 
$C^{\infty}$-solution for suitably small $\eps$ if both the quadratic and 
cubic null conditions are satisfied. 
Note that the (quadratic) null condition is originally introduced by 
Klainerman~\cite{Kla} and Christodoulou~\cite{Chr} in three-dimensional, 
quadratic quasilinear case. 
In 3D case, we do not need structural restrictions on the cubic part, 
while in 2D case, what we can expect under the quadratic null condition only 
(without the cubic one) is the lower estimate for the lifespan $T_{\eps}$ 
in the form $T_{\eps}\ge \exp(c/\eps^2)$ with some $c>0$, 
and this is best possible in general (see \cite{Go} for an example of 
small data blow-up). 
According to \cite{Go}, it holds under the quadratic null condition that
\[
 \liminf_{\eps\to +0} \eps^2\log T_{\eps} 
 \ge 
 \frac{1}{\dis{\sup_{(\rho,\omega) \in \R\times \Sph^1} \left[ 
   -F_{\rm c}(\hat{\omega}) (\mathcal{F}_0{[f,g]}(\rho, \omega))^2
  \right]}}
\]
with the convention $1/0=+\infty$, where 
$\mathcal{F}_0{[f,g]}(\rho, \omega)$ is the radiation field 
associated with $f$, $g$ (see Section~3.5 in \cite{Kat2} or 
Section~6.2 in \cite{Hor_b} for its definition). 
More information on the detailed lifespan estimates and the related topics 
can be found in 
\cite{Ali_book}, \cite{Ali1b}, \cite{De}, \cite{Hor}, 
\cite{Hor_b}, \cite{Hos2}, \cite{J}, \cite{J_b},  \cite{Kat2}, 
\cite{SagSu}, \cite{Su1}, etc., and the references cited therein.

Recently, a lot of efforts have been made for finding weaker structural 
conditions than the null conditions which ensure the small data global 
existence (see e.g., \cite{Ali1a}, \cite{Ali2}, 
\cite{HidYok}, \cite{HidZha}, \cite{Hos3}, 
\cite{Kat2}, \cite{Kat3}, \cite{KK}, 
\cite{KMatoS}, \cite{KMatsS}, \cite{KMuS}, \cite{Kub}, \cite{Lind1}, 
\cite{Lind2}, \cite{LR1}, \cite{LR2}, 
\cite{NS} and so on). 
The Agemi condition, which we are interested in, is one of them. 
From now on we always assume that the quadratic null condition is 
satisfied, and we put $P(\omega)=F_{\rm c}(\hat{\omega})$ for 
$\omega\in \Sph^1$ to focus our attentions on contributions from the 
cubic part of $F$. 
We say that the Agemi condition is satisfied if 
\begin{align}
P(\omega)\ge 0, \quad \omega \in \Sph^1.
\tag{{\bf A}}
\end{align}
It has been shown by Hoshiga~\cite{Hos3} and Kubo~\cite{Kub} that 
%the quadratic null condition and ({\bf A}) imply 
({\bf A}) implies the small data global existence to 
\eqref{Equ}--\eqref{Data} for the positive time direction 
(see also \cite{Su2} and \cite{HNS} for  closely related results for 
the Klein-Gordon and Schr\"odinger equations). 
Moreover, if the inequality in ({\bf A}) is strict, i.e., 
\begin{align}
P(\omega)> 0, \quad \omega \in \Sph^1,
\tag{${\bf A}_+$}
\end{align}
then we have 
\[
\|u(t)\|_{E}=O((\log t)^{-1/4 +\delta})
\] 
as $t \to +\infty$ with arbitrarily small $\delta>0$, where the energy norm 
$\|\phi(t)\|_{E}$ for a function $\phi(t,x)$ is defined by
\[
 \|\phi(t)\|_{E}^2
=\frac{1}{2} \int_{\R^2} |\pa \phi(t,x)|^2\, dx
\]
(see \cite{KMatsS}, \cite{KMuS}). 
It is obvious that ({\bf A}) is weaker than the cubic null condition, and 
a typical example of $F_{\rm c}(\pa u)$ satisfying (${\bf A}_+$) is the cubic 
damping term $-(\pa_t u)^3$. 
Note also that the energy decay of this kind never occur 
under the cubic null condition unless $f=g\equiv 0$ 
(see e.g., Chapter~9 in \cite{Kat2} for the detail). 
Therefore it will be fair to say that (${\bf A}_+$) yields
dissipative structure. 

Now we come to a  question naturally: 
{\em Does the energy decay occur when $({\bf A})$ is satisfied but 
the cubic null condition and $({\bf A}_+)$ are violated?} 
To the authors' knowledge, there are no previous works which address 
this question except \cite{NS} (see Remark~\ref{rmk_system} below), 
and it is far from trivial because the expected dissipation 
is weaker. 
The aim of this article is to give an affirmative answer to this question 
in the single case. 
The main result is as follows.

\begin{thm}  \label{Thm_main}%-------------------------
Assume that quadratic null condition and $({\bf A})$ are satisfied 
but the cubic  null condition is violated. 
For the global solution $u$ 
to \eqref{Equ}--\eqref{Data}, there exist positive constants $C$ and 
$\lambda$ such that 
\[
\|u(t)\|_E\le \frac{C\eps}{(1+ \eps^2\log (t+2))^{\lambda}}
\]
for $t\ge 0$, provided that $\eps$ is sufficiently small. 
\end{thm}
%---------------------------------------------

%-----------------------
\begin{rmk}\label{rmk_ex1}
The cubic nonlinear terms below are examples of $F_{\rm c}(\pa u)$ 
which satisfy ({\bf A}) but fail to satisfy the cubic null condition 
and (${\bf A}_+$): 
\begin{align*}
-(\pa_{1} u)^2\pa_t u, \qquad 
-(\pa_{1}u)^2(\pa_{t}u+\pa_{2}u), \qquad 
-(\pa_t u+\pa_{2} u)^3.
\end{align*}
The corresponding $P(\omega)$'s are 
$\omega_1^2$, $\omega_1^2 (1-\omega_2)$, $(1-\omega_2)^3$, 
respectively. We will give more precise estimates of $\lambda$ 
for these three cases in Section~\ref{sec_rem_lambda}.
\end{rmk}
%-----------------------

%-----------------------
\begin{rmk}
An analogous result for cubic derivative nonlinear Schr\"odinger equations 
in $\R^1$ will be presented in the forthcoming paper \cite{LNSS3}.
\end{rmk}
%-----------------------

%-----------------------
\begin{rmk}\label{rmk_system}
It may be natural to ask what happens in the system case, since the 
null conditions and the Agemi-type structural conditions can be defined also 
for systems (see \cite{KMatsS}). 
Concerning this question, two of the authors have pointed out in \cite{NS} 
that the situation is much more delicate than the single case 
by considering the two-component system
\begin{align*}
\left\{\begin{array}{l}
 \Box u_1=- (\pa_t u_2)^2 \pa_t u_1,\\
 \Box u_2=- (\pa_t u_1)^2 \pa_t u_2.
 \end{array}\right.
\end{align*}
See \cite{NS} for the details, and also \cite{LNSS1}, \cite{LNSS2} for 
closely related results on a system of nonlinear Schr\"odinger equations. 
It seems that there are many interesting problems to be studied 
in systems with weakly dissipative structure. 
\\
\end{rmk}
%-----------------------

We close the introduction with the contents of this article. 
In the next section we introduce a detailed pointwise estimate for the 
solution to \eqref{Equ}--\eqref{Data} under the quadratic null condition and 
({\bf A}). 
Section~\ref{sec_key} is devoted to two important lemmas which play key roles 
in our analysis. After that, Theorem \ref{Thm_main} will be proved in Section
\ref{sec_proof}, and some remarks on the decay rates will be given in 
Section~\ref{sec_rem_lambda}. 
Finally, in Section~\ref{sec_appendix}, we will give an outline of the proof 
of Lemma~\ref{lem_pointwise}.

%%%%%%%%%%%-----------------------------------------------------------------
\section{A detailed pointwise estimate under ({\bf A})}  
\label{sec_sharp_est}
%%%%%%%%%%%-----------------------------------------------------------------
In this section we introduce a detailed pointwise estimate for the solution 
to \eqref{Equ}--\eqref{Data} under the quadratic null condition and ({\bf A}). 
In what follows, we write $\jb{z} =\sqrt{1+|z|^{2}}$ for $z\in \R^d$.
%-----------
\begin{lem} \label{lem_pointwise}
Let $0<\mu <1/10$. Assume that the quadratic null condition and $({\bf A})$ 
are satisfied. If $\eps$ is suitably small, there exists a positive constant 
$C$, not depending on $\eps$, such that the solution $u$ to 
\eqref{Equ}--\eqref{Data} satisfies 
\begin{align}
\label{Decay}
 |\pa u(t,r\omega)|
 \le 
 \frac{C\eps}{\sqrt{t}} 
 \min \left\{ 
 \frac{1}{\sqrt{P(\omega)\eps^2\log t}}, \frac{1}{\jb{t-r}^{1-\mu}} \right\}
\end{align}
for $(t, r,\omega)\in [2,\infty)\times [0,\infty)\times \Sph^1$. 
\end{lem}
%-----------

The estimate \eqref{Decay} is non-trivial, 
but the ideas and tools needed in the proof are essentially not new.
% and a bit tedious. 
So we shall put the proof of Lemma~\ref{lem_pointwise} off 
until the final section, 
and we are getting into the proof of Theorem~\ref{Thm_main} 
by using \eqref{Decay} first.

%%%%%%%%%%%-----------------------------------------------------------------
\section{Key lemmas} \label{sec_key}
%%%%%%%%%%%-----------------------------------------------------------------
This section is devoted to two important lemmas which play key roles 
in our analysis. 
Throughout this section, we suppose that $\Psi(\theta)$ is a real-valued 
function on $[0,2\pi]$ which can be 
written as a (finite) linear combination of the terms 
$\cos^{p_1} \theta \sin^{p_2} \theta$ with $p_1$, $p_2\in \Z_{\ge 0}$. 

%-----------
\begin{lem} \label{lem_key1} 
If $\Psi(\theta)\ge 0$ for all $\theta\in [0,2\pi]$, 
then we have either of the following three assertions:
\begin{itemize}
\item[$({\bf a})$] 
$\Psi(\theta)=0$ for all $\theta\in [0,2\pi]$.
\item[$({\bf b})$] 
$\Psi(\theta)>0$ for all $\theta\in [0,2\pi]$.
\item[$({\bf c})$]
There exist positive integers $m$, $\nu_1, \ldots, \nu_m$, 
points $\theta_1,\ldots, \theta_m \in [0,2\pi]$, 
and 
positive constants $c_1,\ldots, c_m$ such that 
\begin{itemize}
\item[$\bullet$]
$\Psi(\theta)>0$ for 
$\theta\in [0,2\pi]\backslash \{\theta_1,\ldots, \theta_m \}$, 
\item[$\bullet$]
$\Psi(\theta)=(\theta-\theta_j)^{2\nu_j}\bigl(c_j+o(1)\bigr)$
as $\theta\to \theta_j$
for each $j=1,\ldots, m$.
\end{itemize}
\end{itemize}
\end{lem}
%-----------
\begin{proof}
Let $\mathcal{N}$ be the set of zeros of $\Psi$ on $[0,2\pi]$. 
It is easy to see that the case $\mathcal{N}=\emptyset$ corresponds to 
the case ({\bf b}) in the statement. 
Next we consider the case of $\sharp \mathcal{N}=\infty$. 
It follows from the Bolzano-Weierstrass theorem that $\mathcal{N}$ has an 
accumulation point. This is impossible unless $\Psi$ vanishes identically on 
$[0,2\pi]$ since $\Psi$ is a trigonometric polynomial. Therefore
we have ({\bf a}). 
What remains is the case where $0<\sharp\mathcal{N}<\infty$. 
In this case we can write 
$\mathcal{N}$ as $\{\theta_1,\ldots, \theta_m\}$ with $m=\sharp\mathcal{N}$. 
Note that $\Psi(\theta)>0$ for $\theta\in [0,2\pi]\backslash \mathcal{N}$. 
Now let us focus on local behavior of $\Psi(\theta)$ near the point 
$\theta_j$. We observe that we can take $\kappa_j\in \Z_{>0}$ such that
$\Psi^{(l)}(\theta_j)=0$ for $l\le \kappa_j-1$ and 
$\Psi^{(\kappa_j)}(\theta_j)\ne 0$. 
By the Taylor expansion, we have 
\begin{align*}
\Psi(\theta)
=
\sum_{l\le \kappa_j} \frac{\Psi^{(l)}(\theta_j)}{l!}(\theta-\theta_j)^{l}
+O((\theta-\theta_j)^{\kappa_j+1})
=
(\theta-\theta_j)^{\kappa_j}
\left(\frac{\Psi^{(\kappa_j)}(\theta_j)}{\kappa_j!}
+o(1)\right)
\end{align*}
as $\theta\to \theta_j$. By the assumption that $\Psi$ is non-negative, we 
see that $\kappa_j$ must be an even integer and 
$\Psi^{(\kappa_j)}(\theta_j)$ must be strictly positive. 
Therefore we arrive at the case ({\bf c}) by setting 
$c_j=\Psi^{(\kappa_j)}(\theta_j)/(\kappa_j!)$ and $\nu_j=\kappa_j/2$.
 \end{proof}

%-----------
\begin{lem}  \label{lem_key2}
Assume that $\Psi$ satisfies $({\bf c})$. 
We set $\nu=\max\{\nu_1,\ldots, \nu_m\}$. 
Then, for $0<\gamma <1/(2\nu)$, we have 
\[
 \int_0^{2\pi} \frac{d\theta}{\Psi(\theta)^{\gamma}}<\infty.
\]
\end{lem}
%-----------
\begin{proof}
We consider only the case where $\theta_j\ne 0$, $2\pi$ for $j=1,\ldots, m$. 
The other case can be also shown by minor modifications.
We take positive constants $\delta_j$ ($j=1,\ldots, m$) so small that the 
intervals $J_j=(\theta_j-\delta_j, \theta_j+\delta_j)$ 
satisfy
\[
J_j\cap J_k=\emptyset
\quad \mbox{for}\quad 
1\le j< k\le m
\] 
and
\[
\Psi(\theta)\ge \frac{c_j}{2} (\theta-\theta_j)^{2\nu_j} 
\quad \mbox{for}\quad
\theta\in {J_j}.
\] 
We also set 
\[
 K=[0,2\pi]\backslash \bigcup_{j=1}^{m}J_j.
\]
Since $K$ is compact, we can take 
$M> 0$ such that $\Psi(\theta)\ge M$ for $\theta\in K$. So it follows that
\[
 \int_K \frac{d\theta}{\Psi(\theta)^{\gamma}}
 \le 
 \frac{2\pi}{M^{\gamma}}
 <
 \infty.
\]
On the other hand, since $2\gamma\nu_j\le 2\gamma\nu < 1$, we have 
\[
 \int_{J_j} \frac{d\theta}{\Psi(\theta)^{\gamma}}
\le 
\left(\frac{2}{c_j}\right)^{\gamma}
\int_{-\delta_j}^{\delta_j} \frac{d\theta}{|\theta|^{2\gamma\nu_j}}
<\infty 
\]
for $j=1,\ldots, m$. Summing up, we arrive at 
\[
  \int_0^{2\pi} \frac{d\theta}{\Psi(\theta)^{\gamma}}
=
\int_K \frac{d\theta}{\Psi(\theta)^{\gamma}} 
+ 
 \sum_{j=1}^{m}\int_{J_j} \frac{d\theta}{\Psi(\theta)^{\gamma}}
<\infty,
\]
as desired.
\end{proof}

%%%%%%%%%%%-----------------------------------------------------------------
\section{Proof of Theorem~\ref{Thm_main}}  \label{sec_proof}
%%%%%%%%%%%-----------------------------------------------------------------
We are ready to prove Theorem~\ref{Thm_main}. 
In what follows, we denote various positive constants by the same letter $C$ 
which may vary from one line to another. 
Since $f$ and $g$ are compactly-supported, we can 
take $R>0$ such that 
\begin{align}
\supp f \cup \ \supp g \subset \{ x \in \R^2 ; |x| \le R \}.
\label{cpt_supp}
\end{align}
Then, by the finite propagation property, we have
\begin{align}
 \supp u(t,\cdot) \subset \{ x \in \R^2 ; |x| \le t+R \}
\label{finite_propa}
\end{align}
for $t\ge 0$. 

As mentioned in the introduction, we already know that the conclusion is 
true under (${\bf A}_+$). So, in what follows, we assume that the quadratic 
null condition and (${\bf A}$) are satisfied but the cubic null 
condition and (${\bf A}_+$) are violated. 
Then we see that the case ({\bf a}) and ({\bf b}) in Lemma~\ref{lem_key1}  
are excluded by 
$\Psi(\theta)=P(\cos \theta, \sin \theta)$, whence it satisfies ({\bf c}). 
Therefore, by Lemma~\ref{lem_key2}, there exists  $0<\lambda<1/4$ such that 
\[
\int_{0}^{2\pi}
\frac{ d\theta}{P(\cos \theta, \sin \theta)^{2\lambda}} 
< \infty.
\]
With this $\lambda$, we choose $\mu$ such that 
\[
0<\mu <\min\left\{\frac{1}{10}, \frac{1-4\lambda}{2-4\lambda}\right\}.
\]
Let $t\ge 2$ from now on. By Lemma~\ref{lem_pointwise}, we have
\begin{align}
 |\pa u(t,r\omega)|
 &\le 
 \frac{C\eps }{\sqrt{t}} 
 \left(\frac{1}{\sqrt{P(\omega)\eps^2 \log t}}\right)^{2\lambda}
 \left(\frac{1}{\jb{t-r}^{1-\mu}} \right)^{1-2\lambda}
 \notag\\
 &=
  \frac{C\eps}{(\eps^2\log t)^{\lambda}} 
  \cdot
  \frac{1}{P(\omega)^{\lambda}}
  \cdot
  \frac{1}{\sqrt{t} \jb{t-r}^{(1-\mu)(1-2\lambda)}}
 \label{Decay_interpol}
\end{align}
for $(t, r,\omega)\in [2,\infty)\times (0,\infty)\times \Sph^1$. 
Next we set $\rho=(\eps^2\log t)^{\frac{2\lambda}{1-2\mu}}$. 
For small $\eps>0$, we have $0<\rho<t$, and thus 
$0<t+R-\rho\le t+R$.
Then we can split 
\begin{align*}
2\|u(t)\|_E^2
%% &=\int_{|x|\le t+R} |\pa u(t,x)|^2 dx\\
&=
\int_{|x|\le t+R-\rho} |\pa u(t,x)|^2 dx
+
\int_{t+R-\rho\le |x|\le t+R}|\pa u(t,x)|^2 dx\\
&=:I_1+I_2.
\end{align*}
We also note that 
\[ 
 r/t\le (t+R)/t\le 1+R/2 \ \ \mbox{for}\ \  
0\le r\le t+R, 
\]
and 
\[
0<\rho\le R+t-r\le \sqrt{2}(1+R)\jb{t-r}\ \  \mbox{for}\ \  
0\le r\le t+R-\rho.
\]
By using the polar coordinates,
we deduce from \eqref{Decay} and \eqref{Decay_interpol} that
\begin{align*}
 I_1 
 &\le  \int_{0}^{2\pi}
 \int_0^{t+R-\rho} \left(\frac{C\eps}{\sqrt{t}\jb{t-r}^{1-\mu}}\right)^2 rdr 
 d\theta
 \\
 &\le  C\eps^2\int_0^{t+R-\rho} \frac{r dr}{t\jb{t-r}^{2-2\mu}} 
 \\
 &\le  
  C\eps^2 \int_0^{t+R-\rho} \frac{dr}{(R+t-r)^{2-2\mu}} \\
 &\le  
  \frac{C\eps^2}{\rho^{1-2\mu}}
\end{align*}
and
\begin{align*}
I_2 
&\le 
\frac{C\eps^2}{(\eps^2\log t)^{2\lambda}}
\left(
\int_{0}^{2\pi}\frac{ d\theta}{P(\cos\theta,\sin \theta)^{2\lambda}}
\right)
\left(
\int_{t+R-\rho}^{t+R}
\frac{ r dr}{t\jb{t-r}^{2(1-\mu)(1-2\lambda)}} 
\right)\\
&\le
\frac{C\eps^2}{(\eps^2\log t)^{2\lambda}}
\int_{\R}
\frac{d\sigma}{\jb{\sigma}^{2(1-\mu)(1-2\lambda)}},
\end{align*}
respectively. Since $2(1-\mu)(1-2\lambda)>1$, we see that the integral 
in the last line converges.
Eventually we have 
\[
\|u(t)\|_E^2
\le 
\frac{C\eps^2}{\rho^{1-2\mu}}+\frac{C\eps^2}{(\eps^2\log t)^{2\lambda}}
\le
\frac{C\eps^2}{(\eps^2\log (t+2))^{2\lambda}}.
\]
Also we obtain 
\[
\|u(t)\|_E^2
\le
C\eps^2 \int_{0}^{t+R} \frac{r dr}{t\jb{t-r}^{2-2\mu}}
\le
C\eps^2 \int_{\R} \frac{d\sigma}{\jb{\sigma}^{2-2\mu}}
\le 
C\eps^2
\]
by \eqref{Decay}. Summing up, we arrive at the desired estimate.
\qed\\

%%%%%%%%%%%-----------------------------------------------------------------
\section{Remarks on the decay rates}  \label{sec_rem_lambda}
%%%%%%%%%%%-----------------------------------------------------------------
It is worthwhile to mention the exponent $\lambda$ appearing in 
Theorem~\ref{Thm_main}. 
In view of the argument in Section~\ref{sec_proof}, 
we can see that $\lambda$ is determined 
by $\nu$ coming from Lemma~\ref{lem_key2}. 
To be more precise, we can take 
$\lambda=1/(4\nu) -\delta$ with arbitrarily small $\delta>0$, 
and $2\nu$ is the maximum of the vanishing order of zeros of 
$\Psi(\theta)=P(\cos \theta,\sin \theta)$. 

Now, let us compute $\nu$ for the examples of $F_{\rm c}(\pa u)$ raised 
in Remark~\ref{rmk_ex1}. 

\begin{itemize}
%-----------
\item[(1)] 
We first focus on $F_{\rm c}(\pa u)=-(\pa_{1}u)^2\pa_{t}u$. 
Since $\Psi(\theta)=\cos^2\theta$, we can check that
\begin{align*}
 &\Psi(\theta)=(\theta -\pi/2)^2 (1+o(1))
\quad (\theta \to \pi/2),
\\
 &\Psi(\theta)=(\theta -3\pi/2)^2 (1+o(1))
\quad (\theta \to 3\pi/2),
\end{align*}
and $\Psi(\theta)>0$ when $\theta\ne \pi/2$, $3\pi/2$. These tell us 
that $\nu=1$, and thus we have 
$\|u(t)\|_{E}=O((\log t)^{-1/4 +\delta})$ as $t \to \infty$,
where $\delta>0$ can be arbitrarily small.
%-----------
\item[(2)] 
In the case of $F_{\rm c}(\pa u)=-(\pa_{1}u)^2(\pa_{t}u+\pa_{2}u)$, 
we see that 
$\Psi(\theta)=\cos^2\theta (1-\sin\theta)$, and its zeros are 
$\theta=\pi/2$ and $3\pi/2$. Near these points, we have
\[
 \Psi(\theta)=(\theta -\pi/2)^4 (1/2+o(1))
\quad (\theta \to \pi/2)
\]
and 
\[
 \Psi(\theta)=(\theta -3\pi/2)^2 (2+o(1))
\quad (\theta \to 3\pi/2).
\]
Hence $\nu=\max\{2,1\}=2$, from which it follows that 
$\|u(t)\|_{E}=O((\log t)^{-1/8 +\delta})$ as $t \to \infty$
with arbitrarily small $\delta>0$.
%-----------
\item[(3)] 
For $F_{\rm c}(\pa u)=-(\pa_t u+\pa_{2} u)^3$, we have 
$\Psi(\theta)=(1-\sin\theta)^3$. This vanishes only when $\theta=\pi/2$, 
and it holds that 
\[
 \Psi(\theta)=(\theta -\pi/2)^6 (1/8+o(1))
\quad (\theta \to \pi/2).
\]
Therefore $\nu=3$, and this implies that 
$\|u(t)\|_{E}$ decays like $O((\log t)^{-1/12 +\delta})$ as 
$t\to \infty$ with $0<\delta \ll 1/12$.
%-----------
\end{itemize}

%-----------
\begin{rmk}
It is not certain whether these decay rates are the best or not. 
It may be an interesting problem to specify the optimal rates for the 
energy decay.
\end{rmk}
%-----------

%%%%%%%%%%%-----------------------------------------------------------------
%\appendix
\section{Proof of Lemma~\ref{lem_pointwise}}  \label{sec_appendix}
%%%%%%%%%%%-----------------------------------------------------------------
What remains is the proof of Lemma~\ref{lem_pointwise}, whose 
outline will be given in this section. 
We note that many parts of the argument below are almost the same as those in 
the previous works \cite{KMatoS}, \cite{KMatsS}, \cite{KMuS}, \cite{NS}, etc., 
but we must be more careful in some parts. We divide the argument into several 
small steps.\\

%%%%%%%%%%-------------------------
\noindent \underline{\bf Step 1:}\ 
%%%%%%%%%%%-------------------------
First we remember a sort of the commuting vector fields technique 
which we need. We introduce
\begin{align*}
 S := t \pa_{t} + x_{1}\pa_{1} + x_{2}\pa_{2}, \ 
 L_{1} := t \pa_{1} + x_{1} \pa_{t}, \ 
 L_{2} := t \pa_{2} + x_{2} \pa_{t}, \ 
 \Omega := x_1 \pa_2 - x_2 \pa _1,
\end{align*}
and we set
$\Gamma =(\Gamma_{j})_{0 \le j \le 6}
=(S , L_{1} , L_{2} , \Omega , \pa_{0} , \pa_{1} , \pa_{2})$.
For a multi-index 
$\alpha = (\alpha_{0},\alpha_{1},\cdots,\alpha_{6}) \in \Z^{7}_{\ge 0}$, 
we write 
$|\alpha |=\alpha_{0}+\alpha_{1}+\cdots+\alpha_{6}$ 
and 
$\Gamma^{\alpha}
=\Gamma_{0}^{\alpha_{0}}\Gamma_{1}^{\alpha_{1}}\cdots \Gamma_{6}^{\alpha_{6}}$.
We define $|\, \cdot\,|_s$ by 
\begin{align*}
 |\phi (t,x)|_{s}
 =
 \sum _{|\alpha |\le s}|\Gamma^{\alpha }\phi (t,x)|.
\end{align*}
We write $r=|x|$ and $\mathcal{D}=(\pa_r-\pa_t)/2$.
We also set 
$\Lambda_{\infty}:=\{ (t,x) \in [0,\infty)\times \R^2 ; |x|\geq t/2\geq 1 \}$. 
Then we have 
\begin{align}
\left| 
|x|^{1/2}\pa\phi(t,x)-\hat \omega(x) \mathcal{D}\left( |x|^{1/2}\phi(t,x) \right) 
\right|
\le C \jb{t+|x|}^{-1/2}| \phi(t,x) |_1
\label{est_U_to_u}
\end{align}
for $(t,x)\in \Lambda_\infty$, where 
$\hat{\omega}(x)=(-1,x_1/|x|, x_2/|x|)$.  
See Corollary 3.1 in \cite{KMuS} for the proof.\\

%%%%%%%%%%-------------------------
\noindent \underline{\bf Step 2:}\ 
%%%%%%%%%%%-------------------------
Next we make some reductions, whose essential idea goes back to John~\cite{J} 
and H\"ormander~\cite{Hor}. 
Let $u$ be the solution to \eqref{Equ}--\eqref{Data} on 
$[0,\infty) \times \R^2$.
We define $U(t,x)=\mathcal{D}(|x|^{1/2}u(t,x))$. 
We also introduce $H(t,x)$ by
\[
\begin{array}{l}
H
=
-\dfrac{1}{2}
\left( r^{1/2}F(\pa u) - \dfrac{1}{t}P(\omega)U^3 \right) 
- \dfrac{1}{8r^{3/2}}(4\Omega^2+1)u.
\end{array}
\]
From the relation
\[
(\pa_t+\pa_r)(\pa_t-\pa_r)(r^{1/2}\phi)
 =
 r^{1/2}\Box \phi+\frac{1}{4r^{3/2}}(4\Omega^{2}+1)\phi,
\]
it follows that
\begin{align}\label{profile}
 (\pa_t+\pa_r) U(t,x)=-\dfrac{P(\omega)}{2t} U(t,x)^3  + H(t,x).
\end{align}
The basic estimates for $H$ and $u$ can be summarized in the following lemmas. 
%------------------------------------------------------------------------
\begin{lem} Under that the quadratic null condition and \eqref{finite_propa}, 
we have
\begin{align}
 |H(t,x)|
\le 
C t^{-1/2}\bigl(1+t|\pa u|^2 +|u|_{\sharp,0}\bigr) |u|_{\sharp,0} |u|_1
+C t^{-3/2}|u|_2
\label{est_H}
\end{align}
for 
$(t,x)\in\Lambda_{\infty,R}:=
\{(t,x)\in \Lambda_{\infty}\,;\, |x|\le t+R \}$, 
where $R$ comes from \eqref{cpt_supp}, and 
\begin{align*}
 |u (t,x)|_{\sharp,0}
 =
 |\pa u (t,x)| + \jb{t+|x|}^{-1}|u (t,x)|_{1}.
\end{align*}
\label{remainder}
\end{lem}
%-------------------------------------------------------------------------

%---------------------------------------------------------------------------
\begin{lem}
Let $0<\mu <1/10$. 
If the quadratic null condition and $({\bf A})$ are satisfied and 
$\eps$ is suitably small, 
then the solution $u$ to \eqref{Equ}--\eqref{Data} satisfies
\begin{align}
|\pa u(t,x)| \le C \eps \jb{t+|x|}^{-1/2}\jb{t-|x|}^{\mu-1}
\label{apriori1}
\end{align}
and
\begin{align}
|u(t,x)|_{2} \le C \eps \jb{t+|x|}^{-1/2+\mu}
\label{apriori2}
\end{align}
for $(t,x) \in [0,\infty)\times \R^2$.
\end{lem}
%---------------------------------------------------------------------------
For the proof, see Lemma 2.8 in \cite{KMatsS} 
and Section 3 in \cite{KMatsS}, respectively.
From \eqref{est_U_to_u},  \eqref{est_H}, \eqref{apriori1} and \eqref{apriori2}, we have
\begin{align}
|U(t,x)|
\le 
 \left| |x|^{1/2}\pa u(t,x) \right|
 + 
 \left| |x|^{1/2}\pa u(t,x)-\hat \omega U(t,x) \right|
\le 
C\eps \jb{t-|x|}^{\mu-1}
\label{est_U}
\end{align}
and
\begin{align}
|H(t,x)|
&\le 
C\eps^{2}t^{-1/2}\jb{t+|x|}^{\mu-1}\jb{t-|x|}^{\mu-1} 
+ C\eps t^{-3/2} \jb{t+|x|}^{\mu-1/2} \notag \\
&\le 
C\eps t^{2\mu-3/2}\jb{t-|x|}^{-\mu-1/2}
\label{H}
\end{align} 
for $(t,x) \in \Lambda_{\infty,R}$. 
%%
%%Note that the weights $|x|^{-1}$, $t^{-1}$, $(1+t)^{-1}$, 
%%$\jb{t+|x|}^{-1}$ are equivalent to each other on $\Lambda_{\infty,R}$. 
%%
Now we set $V(t;\sigma,\omega)=U(t,(t+\sigma)\omega)$ and 
$G(t;\sigma,\omega)=H(t,(t+\sigma)\omega)$
for $(t;\sigma,\omega)\in [t_{0,\sigma},\infty)\times \R\times \Sph^1$, 
where $t_{0,\sigma}=\max\{ 2 , -2\sigma \}$. 
Then we see that  \eqref{profile} is reduced to 
\begin{align}
 \pa_t V(t)  =  -\dfrac{P(\omega)}{2t}V(t)^3+ G(t),
 \label{profileV}
\end{align} 
and we observe that 
\begin{align*}
\Lambda_{\infty,R}=\bigcup_{(\sigma,\omega)
\in(-\infty,R]\times\Sph^1} 
\{ \ (t,(t+\sigma)\omega)\ ;\ t\geq t_{0,\sigma} \}.
\end{align*}
We also note that there exists a positive constant 
$c_0$ depending only on $R$ such that
\begin{align}
 \jb{\sigma}/c_0 \le t_{0,\sigma}\le c_0 \jb{\sigma}
\label{t_0}
\end{align}
for $\sigma\in (-\infty,R]$. 
It follows from \eqref{est_U} and \eqref{H} that
\begin{align}\label{V}
|V(t;\sigma,\omega)| 
\le 
C\eps \jb{\sigma}^{\mu-1} 
\end{align}
and
\begin{align}\label{K}
|G(t;\sigma,\omega)| 
\le 
C\eps \jb{\sigma}^{-\mu-1/2} t^{2\mu-3/2} 
\end{align}
for 
$(t,\sigma,\omega)\in [t_{0,\sigma},\infty)\times (-\infty, R]\times \Sph^1$.
\\

%%%%%%%%%%-------------------------
\noindent \underline{\bf Step 3:}\ 
%%%%%%%%%%%-------------------------
Let us also recall the following useful lemma due to Matsumura.
 
%----------------------------------------------------------------------------
\begin{lem}\label{lem_Matsumura}
Let $C_0>0$, $C_1\ge 0$, $p>1$, $q>1$ and $t_0\ge2$. 
Suppose that a function $\Phi(t)$ satisfies 
\begin{align*}
\frac{d\Phi}{dt}(t)
\le -\frac{C_0}{t}\left|\Phi(t)\right|^p + \frac{C_1}{t^{q}} 
\end{align*}
for $t\ge t_0$. Then we have
\begin{align*}
\Phi(t)\le \frac{C_2}{(\log t)^{p^{*}-1}}
\end{align*}
for $t\ge t_0$, where $p^{*}$ is the H\"older conjugate of $p$ (i.e., 
$1/p+1/p^{*}=1$), and
\begin{align*}
C_2=\frac{1}{\log 2}\left( (\log t_0)^{p^{*}}\Phi(t_0) 
  + C_1\int_{2}^{\infty}\frac{(\log \tau)^{p^{*}}}{\tau^{q}}d\tau \right) 
  + \left( \frac{p^{*}}{C_0 p} \right)^{p^{*}-1}.
\end{align*}
\end{lem}
%----------------------------------------------------------------------------

For the proof, see Lemma 4.1 of \cite{KMatsS}. 
\\

%%%%%%%%%%-------------------------
\noindent \underline{\bf Final step:}\ 
%%%%%%%%%%%-------------------------
Now we are in a position to reach \eqref{Decay}. 
Let $(\sigma,\omega)\in (-\infty,R]\times \Sph^1$ be fixed, and we set 
$\Phi(t)=\Phi(t;\sigma,\omega)=P(\omega)V(t;\sigma,\omega)^2$ 
for $t\ge t_{0,\sigma}$. It follows from 
\eqref{profileV}, \eqref{V} and \eqref{K} that 
\begin{align*}
 \pa_t \Phi(t)  
 &=
 2P(\omega)V(t)\pa_t V(t)\\
 &=  
 -\dfrac{P(\omega)^2}{t}V(t)^4+ 2P(\omega)V(t)G(t)\\
 &\le
 -\dfrac{1}{t}\Phi(t)^2+ \frac{C_*\eps^2}{t^{3/2-2\mu} \jb{\sigma}^{3/2}}
\end{align*} 
with some $C_*>0$ not depending on $\sigma$, $\omega$ and $\eps$. 
Therefore we can apply Lemma~\ref{lem_Matsumura} with $p=2$, $q=3/2-2\mu$ and 
$t_0=t_{0,\sigma}$ to obtain 
\[
 0 \le \Phi(t;\sigma,\omega)\le \frac{M(\sigma,\omega)}{\log t},
\]
where 
\[
M(\sigma,\omega)
 =
 \frac{1}{\log 2}\left( 
  (\log t_{0,\sigma})^{2}P(\omega)V(t_{0,\sigma};\sigma,\omega)^2 
  + 
 \frac{C_*\eps^2}{\jb{\sigma}^{3/2}}
 \int_{2}^{\infty}\frac{(\log \tau)^{2}}{\tau^{3/2-2\mu}}d\tau 
 \right) 
 + 1.
\]
By virtue of \eqref{t_0} and \eqref{V}, we see that $M(\sigma,\omega)$ 
can be dominated by a positive constant 
not depending on $\sigma$, $\omega$ and $\eps$. Therefore we deduce that 
\[
 |V(t;\sigma,\omega)|\le \sqrt{\frac{\Phi(t;\sigma,\omega)}{P(\omega)}} 
 \le 
 \frac{C}{\sqrt{P(\omega)\log t}}
\]
for 
$(t,\sigma,\omega)\in [t_{0,\sigma},\infty)\times (-\infty,R]\times \Sph^1$. 
By \eqref{est_U_to_u} and \eqref{apriori2}, we have 
\begin{align*}
 r^{1/2}|\pa u(t,r\omega)|
&\le
\sqrt{2}|V(t;r-t,\omega)|
 + 
 \left| r^{1/2}\pa u(t,r\omega)-\hat \omega U(t,r\omega) \right|\\
&\le 
\frac{C}{\sqrt{P(\omega)\log t}}+ \frac{C\eps}{\jb{t+r}^{1-\mu}}
\end{align*}
for $(t,r\omega)\in\Lambda_{\infty,R}$, whence
\[
 |\pa u(t,r\omega)|
\le 
\frac{C}{\sqrt{rP(\omega)\log t}}  
\left(1 + \frac{\eps\sqrt{P(\omega)\log t}}{t^{1-\mu}}\right)
\le
 \frac{C\eps}{\sqrt{t}}\cdot \frac{1}{\sqrt{P(\omega)\eps^2\log t}}
\]
for $(t,r\omega)\in\Lambda_{\infty,R}$. 
Piecing together this with \eqref{apriori1}, 
we arrive at the desired estimate in the case of 
$(t,r\omega)\in\Lambda_{\infty,R}$. 
It is much easier to derive the bound for $|\pa u(t,r\omega)|$ in the case of 
$(t,r\omega)\not\in\Lambda_{\infty,R}$ (indeed it follows from 
\eqref{apriori1} only), so we skip it here.
\qed

%--------------------------------------------------------------------
\medskip
\subsection*{Acknowledgments}
The authors would like to thank Professor Soichiro Katayama and 
Dr.Yuji Sagawa for their useful conversations on this subject. 
The work of H.~S. is supported by Grant-in-Aid for Scientific Research (C) 
(No.~17K05322), JSPS. 
%--------------------------------------------------------------------

%%%%%%%%%%%%%%%%%%%%%%%%%%%%%%%%%%%%%%%%%%%%%%%%%%%%%%%%%%%%%%%%%%%%%%%%%%%%%%%

%%%%%%%%%%%%%%%%%%%%%%%%%%%%%%%%%%%%%%%%%%%%%%%%%%%%%%%%%%%%%%%%%%%%%%%%%%%%%%%
\end{document}